\documentclass[11pt,a4paper]{amsart} 

\usepackage{latexsym}

\usepackage{geometry}
\usepackage{pinlabel}         
\usepackage{graphicx}
\usepackage{amssymb, amsthm}
\usepackage{epstopdf}
\usepackage{romannum}
\usepackage{tikz}
\usepackage{subfig}
\usepackage{cancel}

  	\newcommand{\Z}{\ensuremath{\mathbb{Z}}}

\theoremstyle{plain}
\newtheorem*{NewDefinition}{Definition}
\newtheorem*{NewTheoremA}{Theorem A}
\newtheorem*{NewTheoremB}{Theorem B}
\newtheorem*{NewPropositionC}{Proposition C}

\newtheorem{theorem}{Theorem}[section]
\newtheorem{lemma}[theorem]{Lemma}

\newtheorem{corollary}[theorem]{Corollary}
\newtheorem{example}[theorem]{Example}
\newtheorem{proposition}[theorem]{Proposition}
\newtheorem{definition}[theorem]{Definition}

\title{On coherence of graph products and Coxeter groups}
\author{Olga Varghese}
\date{\today}
\address{Olga Varghese\\
Department of Mathematics\\
M\"unster University\\ 
Einsteinstra\ss e 62\\
48149 M\"unster (Germany)}
\email{olga.varghese@uni-muenster.de}

\begin{document}

\pagenumbering{arabic}
\begin{abstract}
We study coherence of graph products and Coxeter groups and obtain many results in this direction. 
\end{abstract}
\maketitle

\section{Introduction}
This article is located in the area of group theory. One interesting algebraic property of groups is the coherence. A group is called coherent if every finitely generated subgroup is finitely presented. Classical examples of coherent groups are free groups and free abelian groups. The standard example of an incoherent group is the direct product $F_2\times F_2$ of two free groups.
We are interested in understanding which graph products, Artin and Coxeter groups are coherent. More precisely, let $\Gamma=(V,E)$ be a finite simplicial non-empty graph with vertex set $V$ and edge set $E$. A vertex labeling on $\Gamma$  is a map $\varphi:V\rightarrow \left\{\text{non-trivial finitely generated abelian groups}\right\}$ and an edge labeling on $\Gamma$ is a map $\psi:E\rightarrow\mathbb{N}-\left\{0,1\right\}$. A graph $\Gamma$ with a vertex and edge labeling is called a vertex-edge-labeled graph.
\begin{NewDefinition}
Let $\Gamma$ be a vertex-edge-labeled graph. 
\begin{enumerate}
\item[(i)] If $\psi(E)=\left\{2\right\}$, then $\Gamma$ is called a graph product graph. The graph product $G(\Gamma)$ is the group obtained from the free product of
the $\varphi(v)$, by adding the commutator relations $[g,h] = 1$ for all $g\in \varphi(v)$, $h\in\varphi(w)$ such that $\left\{v,w\right\}\in E$.
\item[(ii)] If $\varphi(V)=\left\{\Z\right\}$, then $\Gamma$ is called an Artin graph and the corresponding Artin group $A(\Gamma)$ is given by the presentation
\[
A(\Gamma)=\langle V\mid \underbrace{vwvw\ldots}_{\psi(\left\{v,w\right\})-\text{letters}}=\underbrace{wvwv\ldots}_{\psi(\left\{v,w\right\})-\text{letters}}\text{ if }\left\{v,w\right\}\in E\rangle
\]
If $\Gamma$ is an Artin graph and $\psi(E)=\left\{2\right\}$, then $\Gamma$ is called a right angled Artin graph and the Artin group $A(\Gamma)$ is called right angled Artin group.
\item[(iii)] If $\varphi(V)=\left\{\Z_2\right\}$, then $\Gamma$ is called a Coxeter graph and the corresponding Coxeter group $C(\Gamma)$ is given by the presentation
\[
C(\Gamma)=\langle V\mid v^2, (vw)^{\psi(\left\{v, w\right\})}\text{ if }\left\{v,w\right\}\in E\rangle
\]
If $\Gamma$ is a Coxeter graph and $\psi(E)=\left\{2\right\}$, then $\Gamma$ is called a right angled Coxeter graph and the Coxeter group $C(\Gamma)$ is called right angled Coxeter group.
\end{enumerate}
\end{NewDefinition}
The first examples to consider are the extremes. If $\Gamma$ is a discrete vertex-edge-labeled graph, then $G(\Gamma)$ is a free product of finitely generated abelian groups. In particular, if $\Gamma$ is a discrete Artin graph with $n$ vertices, then $A(\Gamma)$ is the free group $F_n$ of rank $n$. On the other hand, if $\Gamma$ is a complete graph product graph, then $G(\Gamma)$ is a finitely generated abelian group. The corresponding Coxeter group of the following Coxeter graph is the symmetric group ${\rm Sym}(5)$ and the corresponding Artin group of the following Artin graph is the braid group $B_5$ on $5$ strands .

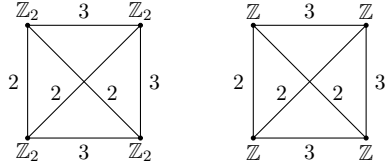
\begin{figure}[h]
\begin{center}
\begin{tikzpicture}[scale=0.75, transform shape]
\draw[fill=black]  (0,0) circle (1pt);
\draw[fill=black]  (2,0) circle (1pt);
\draw[fill=black]  (2,2) circle (1pt);
\draw[fill=black]  (0,2) circle (1pt);
\draw (0,0)--(2,0);
\draw (0,0)--(0,2);
\draw (2,0)--(2,2);
\draw (0,2)--(2,2);
\draw (0,0)--(2,2);
\draw (2,0)--(0,2);

\node at (0,-0.25) {$\Z_2$}; 
\node at (2,-0.25) {$\Z_2$}; 
\node at (0, 2.25) {$\Z_2$}; 
\node at (2,2.25) {$\Z_2$}; 

\node at (1,-0.25) {$3$}; 
\node at (1, 2.25) {$3$}; 
\node at (2.25, 1) {$3$}; 
\node at (-0.25,1) {$2$}; 
\node at (0.5, 0.8) {$2$}; 
\node at (1.5,0.8) {$2$};

\draw[fill=black]  (4,0) circle (1pt);
\draw[fill=black]  (6,0) circle (1pt);
\draw[fill=black]  (6,2) circle (1pt);
\draw[fill=black]  (4,2) circle (1pt);
\draw (4,0)--(6,0);
\draw (4,0)--(4,2);
\draw (6,0)--(6,2);
\draw (4,2)--(6,2);
\draw (4,0)--(6,2);
\draw (6,0)--(4,2);

\node at (4,-0.25) {$\Z$}; 
\node at (6,-0.25) {$\Z$}; 
\node at (4, 2.25) {$\Z$}; 
\node at (6,2.25) {$\Z$}; 

\node at (5,-0.25) {$3$}; 
\node at (5, 2.25) {$3$}; 
\node at (6.25, 1) {$3$}; 
\node at (3.74,1) {$2$}; 
\node at (4.5, 0.8) {$2$}; 
\node at (5.5,0.8) {$2$};

\end{tikzpicture}
\end{center}
\end{figure}

It is natural to ask which graph products, Artin and Coxeter groups are coherent. For right angled Artin groups this has been answered by Droms \cite{Droms}: A right angled Artin group $A(\Gamma)$ is coherent iff $\Gamma$ has no induced cycle of length $>3$. 
We show that Droms arguments can be extended to a much larger class of groups. 
\begin{NewTheoremA}
Let $\Gamma$ be a graph product graph. If $\Gamma$  has no induced cycle of length $>3$, then $G(\Gamma)$ is coherent.
\end{NewTheoremA}

A right angled Artin group $A(\Gamma)$ where $\Gamma$ has a shape of a cycle of length $>3$ is by Droms result incoherent. Concerning arbitrary graph product $G(\Gamma)$ where $\Gamma$ has a shape of a cycle $>3$ this result does not hold. We prove in Theorem B, that a right angled Coxeter group which is defined via graph with a shape of a cycle of length $>3$ is always coherent.  Let us consider a graph product graph  $\Gamma=(V,E)$ with a shape of a cycle of length $4$ such that $\#\varphi(v)\geq 3$ for all $v\in V$, then $G(\Gamma)$ is incoherent. This follows from the observation, that $G(\Gamma)$ is the direct sum of the free product of opposite vertex groups $G(\Gamma)=(\varphi(v_1)*\varphi(v_3))\times(\varphi(v_2)*\varphi(v_4))$. Since the kernel of the canonical map $\varphi(v_1)*\varphi(v_3)\rightarrow\varphi(v_1)\times\varphi(v_3)$ is a free group of rank $\geq 2$, see \cite[I.1.3 Pr. 4]{Serre} it follows that $F_2\times F_2$ is a subgroup of $G(\Gamma)$. It is known that $F_2\times F_2$ is incoherent, hence $G(\Gamma)$ is incoherent.

Let $\Gamma=(V, E)$ be a graph product graph with a shape of a cycle of length $\geq 5$ such that $\infty>\#\varphi(v)\geq 3$ for all $v\in V$. We do not know if  $G(\Gamma)$ is coherent or not. We only know that $F_2\times F_2$ is not a subgroup of $G(\Gamma)$, because $G(\Gamma)$ is Gromov hyperbolic \cite{Hyperbolic} and it is known that Gromov hyperbolic groups do not contain a copy of $\Z\times\Z$. If $G(\Gamma)$ is incoherent, then we would have the same characterization for coherence of graph products of finite abelian vertex groups with cardinality $\geq 3$  as for right angled Artin groups. But if $G(\Gamma)$ is coherent, then the characterization would be more complicated. 

It was proven by Wise and Gordon \cite{Wise2}, \cite{Gordon} that an Artin group $A(\Gamma)$ is coherent iff 
$\Gamma$ has no induced cycle of length $>3$, 
every complete subgraph of $\Gamma$ with $3$ or $4$ vertices has at most one edge label $>2$
and $\Gamma$ has no induced subgraph of the following shape: 
\newpage
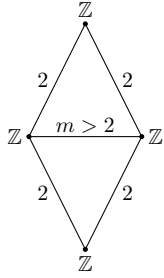
\begin{figure}[h]
\begin{center}
\begin{tikzpicture}[scale=0.75, transform shape]
\draw[fill=black]  (0,0) circle (1pt);
\draw[fill=black]  (2,0) circle (1pt);
\draw[fill=black]  (1,2) circle (1pt);
\draw[fill=black]  (1,-2) circle (1pt);
\draw (0,0)--(1,2);
\draw (0,0)--(2,0);
\draw (0,0)--(1, -2);
\draw (2,0)--(1,2);
\draw (2,0)--(1,-2);

\node at (0.25,1) {$2$}; 
\node at (1.75,1) {$2$}; 
\node at (1,0.2) {$m>2$}; 
\node at (0.25,-1) {$2$}; 
\node at (1.75,-1) {$2$}; 

\node at (-0.25,0) {$\Z$}; 
\node at (2.25,0) {$\Z$}; 
\node at (1,2.25) {$\Z$}; 
\node at (1,-2.25) {$\Z$}; 

\end{tikzpicture}
\end{center}
\end{figure}

Concerning coherence of Coxeter groups we found two results in the literature. A simple criteria for the coherence of Coxeter groups which depends only on the edge labeling and the number of generators was proven by McCammond and Wise
\cite[Theorem 12.2]{WiseCoxeter}: If $\Gamma=(V, E)$ is a Coxeter graph and $\psi(e)\geq \#V$ for all $e\in E$, then $C(\Gamma)$ is coherent. Further, Jankiewicz and Wise proved with probabilistic methods that many infinite Coxeter groups where the Coxeter graph is complete are incoherent  \cite[Theorem 1.2]{Jankiewicz}. 

We present two results regarding coherence of Coxeter groups.
\begin{NewTheoremB}
Let $\Gamma$ be a  Coxeter graph.
\begin{enumerate}
\item[(i)] If $\Gamma$ has no induced cycle of length $>3$ and every complete subgraph is that of a slender Coxeter group  (i.e. every subgroup is finitely generated), then $C(\Gamma)$ is coherent. In particular, if $\Gamma$ has a shape of a tree, then  $C(\Gamma)$ is coherent. 
\item[(ii)] If $\Gamma$ has a shape of a cycle of length $>3$, then $C(\Gamma)$ is coherent.
\end{enumerate}
\end{NewTheoremB}

It is obvious that finite Coxeter groups are slender. Concerning slenderness of infinite Coxeter groups we show:
\begin{NewPropositionC}
Let $C(\Gamma)$ be an infinite Coxeter group. Then $C(\Gamma)$ is slender iff $C(\Gamma)$ decomposes as $C(\Gamma)\cong C(\Gamma_1)\times C(\Gamma_2)$, where $\Gamma_1$, $\Gamma_2$ are induced subgraphs of $\Gamma$ and $C(\Gamma_1)$ is a finite subgroup and $C(\Gamma_2)$ is a finite direct product of irreducible Euclidean reflection groups.
\end{NewPropositionC}
Irreducible Euclidean reflection groups were classified in terms of Coxeter diagrams, see \cite{Humphreys}. 

It follows from Droms characterization that the smallest incoherent right angled Artin graph is a cycle of length $4$. Concerning right angled Coxeter groups we show that a smallest right angled Coxeter graph has 6 vertices and 9 edges.

We prove Theorems A and B by induction on the cardinality of the vertex set of $\Gamma$. Our main technique in the proofs of the main theorems is based on the
following result by Karrass and Solitar \cite[Theorem 8]{Karrass}: An amalgam $A*_C B$ where $A, B$ are coherent groups and $C$ is slender is also coherent.

\section{Graphs}
In this section we briefly present the main definitions and properties
concerning simplicial graphs. For more background results see \cite{Diestel}.

A {\it simplicial graph} is a pair $\Gamma=(V,E)$ of sets such that $V\neq\emptyset$ and $E\subseteq V\times V$. The elements of $V$ are called {\it vertices} and the elements of $E$ are its {\it edges}.  If $V'\subseteq V$ and $E'\subseteq E$, then $\Gamma'=(V', E')$ is called a {\it subgraph} of $\Gamma$. If $\Gamma'$ is a subgraph of $\Gamma$ and $E'$ contains all the edges $\left\{v, w\right\}\in E$ with $v, w\in V'$, then $\Gamma'$ is called an {\it induced subgraph} of $\Gamma$. A subgraph $\Gamma'$ is called {\it proper} if $\Gamma'\neq \Gamma$. A {\it path} of length $n$ is a graph $P_n=(V, E)$ of the form  $V=\left\{v_0, \ldots, v_n\right\}$ and $E=\left\{\left\{v_0, v_1\right\}, \left\{v_1, v_2\right\}, \ldots, \left\{v_{n-1}, v_n\right\}\right\}$ where $v_i$ are all distinct. If $P_n=(V,E)$ is a path of length $n\geq 3$, then the graph $C_{n+1}:=(V, E\cup\left\{\left\{v_n, v_0\right\}\right\})$ is called a {\it cycle} of length $n+1$. A graph $\Gamma=(V,E)$ is called {\it connected} if any two vertices $v, w\in V$ are contained in a subgraph $\Gamma'$ of $\Gamma$ such that $\Gamma'$ is a path. A maximal connected subgraph of $\Gamma$ is called a {\it connected component} of $\Gamma$. A graph $\Gamma$ is called a {\it tree} if $\Gamma$ is a connected graph without induced cycle. A graph is called {\it complete} if there is an edge for every pair of distinct vertices. 

Dirac proved in \cite{Dirac} the following result which we will use in the proofs of Theorem A and Theorem B. 
\begin{proposition}
\label{Dirac}
Let $\Gamma$ be a connected non-complete finite simplicial graph. If $\Gamma$ has no induced cycle of length $>3$, then there exist proper induced subgraphs $\Gamma_1$, $\Gamma_2$ with the following properties:
\begin{enumerate}
\item[(i)] $\Gamma=\Gamma_1\cup\Gamma_2$,
\item[(ii)] $\Gamma_1\cap\Gamma_2$ is complete.
\end{enumerate}
\end{proposition}

\section{Slender and coherent groups}
In this section we present the main definitions and properties concerning slender and coherent groups. We start with the following definition.
\begin{definition}
\begin{enumerate}
\item[(i)] A group $G$ is said to be slender (or Noetherian) if every subgroup of $G$ is finitely generated.
\item[(ii)] A group $G$ is called coherent if every finitely generated group is finitely presented.
\end{enumerate} 
\end{definition}
One can easily verify from the definition of slenderness and coherence that finite groups are slender and coherent. Further, it follows from the classification of finitely generated abelian groups that these groups are  slender and coherent. A standard example of a group which is not slender is a free group $F_n$ for $n\geq 2$. This follows from the fact that the commutator subgroup of $F_n$ for $n\geq 2$ is not finitely generated. 

Concerning slenderness of graph products we want to remark the following result:
\begin{example}
Let $\Gamma=(V,E)$ be a graph product graph such that $\#\varphi(v)\geq 3$ for all $v\in V$ and $G(\Gamma)$ the corresponding graph product.  Then $G(\Gamma)$ is slender iff $\Gamma$ is a complete graph.
\end{example}
\begin{proof}
If $\Gamma$ is a complete graph, then $G(\Gamma)$ is a finitely generated abelian group and hence slender. Let us assume that $\Gamma=(V, E)$  is not complete. Then there exist vertices $v_1, v_2\in V$ such that $\left\{v_1, v_2\right\}\notin E$. Therefore $\Gamma':=(\left\{v_1, v_2\right\}, \emptyset)$ is an induced subgraph and $G(\Gamma')=\varphi(v_1)*\varphi(v_2)$ is a subgroup of $G(\Gamma)$, see \cite{Green}. Since $\#\varphi(v_1)\geq 3$, it follows that the kernel of the natural map  $\varphi(v_1)*\varphi(v_2)\rightarrow\varphi(v_1)\times\varphi(v_2)$ is a free group of rank $\geq 2$. The free group $F_2$ is not slender, therefore $G(\Gamma')$ and $G(\Gamma)$ are not slender. 
\end{proof}

It is not hard to see that the following result is true.
\begin{lemma}
\label{slender}
Let $1\rightarrow G_1\rightarrow G_2\rightarrow G_3\rightarrow 1$ be a short exact sequence of groups. Then $G_2$ is slender iff $G_1$ and $G_3$ are slender. 

In particular, semidirect products of slender groups are slender and finite direct products of slender groups are slender.
\end{lemma}

Concerning slenderness of Coxeter groups we prove:
\begin{NewPropositionC}
Let $C(\Gamma)$ be an infinite Coxeter group. Then $C(\Gamma)$ is slender iff $C(\Gamma)$ decomposes as $C(\Gamma)\cong C(\Gamma_1)\times C(\Gamma_2)$, where $\Gamma_1$, $\Gamma_2$ are induced subgraphs and $C(\Gamma_1)$ is a finite subgroup and $C(\Gamma_2)$ is a finite direct product of irreducible Euclidean reflection groups.
\end{NewPropositionC}
\begin{proof}
If $C(\Gamma)$ is slender, then $F_2$ is not a subgroup of $C(\Gamma)$ and it follows by \cite[17.2.1]{Davis} that $C(\Gamma)$ decomposes as $C(\Gamma)\cong C(\Gamma_1)\times C(\Gamma_2)$, where $C(\Gamma_1)$ is a finite subgroup and $C(\Gamma_2)$ is a finite direct product of irreducible Euclidean reflection groups. 
Now, assume that $C(\Gamma)$ has the above decomposition. 
Let $G$ be an irreducible Euclidean reflection group, then $G$ decomposes as semidirect product of a finitely generated abelian group and a finite group, see \cite{Humphreys}. Since finitely generated abelian groups and finite groups are slender it follows by Lemma \ref{slender} that $G$ is slender. Now we know that $C(\Gamma)$ is a direct product of slender groups. It follows again by Lemma \ref{slender} that $C(\Gamma)$ is slender.
\end{proof}

Since $\Z_2*\Z_2$ is the only irreducible right angled Euclidean reflection group we immediate obtain the following corollary of Proposition C: 
\begin{corollary}
Let $G$ be an infinite right angled Coxeter group. Then $G$ is slender iff there exist $n, k\in\mathbb{N}$ such that $G\cong \Z^n_2\times (\Z_2*\Z_2)^k$. 
\end{corollary}
For example, the corresponding right angled Coxeter groups of the following graphs are slender:

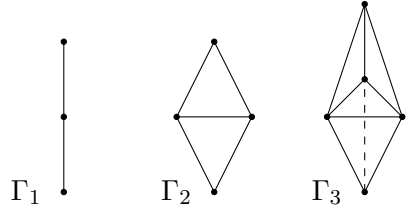
\begin{figure}[h]
\begin{center}
\begin{tikzpicture}
\draw[fill=black]  (0,0) circle (1pt);
\draw[fill=black]  (0,1) circle (1pt);
\draw[fill=black]  (0,2) circle (1pt);
\draw (0,0)--(0,2);
\node at (-0.5,0) {$\Gamma_1$}; 
\draw[fill=black]  (2,0) circle (1pt);
\draw[fill=black]  (1.5, 1) circle (1pt);
\draw[fill=black]  (2.5, 1) circle (1pt);
\draw[fill=black]  (2, 2) circle (1pt);
\draw (2,0)--(1.5,1);
\draw (2,0)--(2.5,1);
\draw (1.5,1)--(2.5,1);
\draw (1.5,1)--(2,2);
\draw (2.5,1)--(2,2);
\node at (1.5,0) {$\Gamma_2$}; 

\draw[fill=black]  (4,0) circle (1pt);
\draw[fill=black]  (3.5, 1) circle (1pt);
\draw[fill=black]  (4.5, 1) circle (1pt);
\draw[fill=black]  (4, 1.5) circle (1pt);
\draw[fill=black]  (4, 2.5) circle (1pt);
\draw (4,0)--(3.5,1);
\draw (4,0)--(4.5,1);
\draw (3.5,1)--(4.5,1);
\draw (3.5,1)--(4,2.5);
\draw (4.5,1)--(4,2.5);

\draw (3.5,1)--(4,1.5);
\draw (4.5,1)--(4,1.5);
\draw (4,1.5)--(4,2.5);
\draw[dashed] (4, 0)--(4,1.5);
\node at (3.5,0) {$\Gamma_3$}; 
\end{tikzpicture}
\end{center}
\end{figure}

\begin{lemma}
\label{slendercoherent}
Let $1\rightarrow G_1\xrightarrow{\iota} G_2\xrightarrow{\pi} G_3\rightarrow 1$ be a short exact sequence of groups. If $G_1$ and $G_2$ are slender and coherent groups, then $G_2$ is slender and coherent. 

In particular, semidirect products of slender and coherent groups are slender and coherent and finite direct products of slender and coherent groups are slender and coherent.
\end{lemma}
\begin{proof}
The slenderness of $G_2$ follows by Lemma \ref{slender}. Let $U$ be a finitely generated subgroup of $G_2$.  We have ${\rm ker}(\pi_{\mid U})\subseteq{\rm im}(G_1)\cong G_1$. The group $G_1$ is slender and coherent, thus  ${\rm ker}(\pi_{\mid U})$ is finitely presented. The quotient $U/{\rm ker}(\pi_{\mid U})$ is isomorphic to the subgroup of $G_3$. We know that $G_3$ is slender and coherent, hence $U/{\rm ker}(\pi_{\mid U})$ is finitely presented. So far we have shown that ${\rm ker}(\pi_{\mid U})$ and $U/{\rm ker}(\pi_{\mid U})$ are finitely presented groups. Now it is not hard to see that $U$ is finitely presented.
\end{proof}
The following consequence of Lemma \ref{slendercoherent} will be need in the proof of Theorem B.
\begin{corollary}
\label{Coxeterslendercoherent}
Let $C(\Gamma)$ be a slender Coxeter group. Then $C(\Gamma)$ is coherent.
\end{corollary}
\begin{proof}
If $C(\Gamma)$ is finite, then it is obvious that $C(\Gamma)$ is coherent. Otherwise by Proposition C follows that $C(\Gamma)$ decomposes as a finite direct product of a finite group and irreducible Euclidean reflection groups. Since an irreducible Euclidean reflection group is a semidirect product of a finitely generated abelian group and a finite group it follows by Lemma \ref{slendercoherent} that this group is slender and coherent.
Thus, $C(\Gamma)$ is a finite direct product of slender and coherent groups and by Lemma \ref{slendercoherent} follows that $C(\Gamma)$ is also coherent.
\end{proof}

The crucial argument in the proofs of the main theorems is the following result which was proven by Karrass and Solitar.
\begin{proposition}(\cite[Theorem 8]{Karrass})
\label{Karrass}
Let $A*_C B$ be an amalgam. If $A$ and $B$ are coherent groups and $C$ is slender, then $A*_C B$ is coherent. 

In particular, coherence is preserved under taking free products.
\end{proposition}
 A direct consequence of the result of Karrass and Solitar is the next corollary:
\begin{corollary}
\label{amalgam}
\begin{enumerate}
\item[(i)]  Let $\Gamma$ be a graph product graph and $\Gamma_1, \Gamma_2$ induced subgraphs such that $\Gamma=\Gamma_1\cup\Gamma_2$. If  $G(\Gamma_1)$ and $G(\Gamma_2)$ are coherent and $G(\Gamma_1\cap\Gamma_2)$ is slender, then $G(\Gamma)$ is coherent.
\item[(ii)] Let $\Gamma$ be a Coxeter graph and $\Gamma_1, \Gamma_2$ induced subgraphs such that $\Gamma=\Gamma_1\cup\Gamma_2$. If  $C(\Gamma_1)$ and $C(\Gamma_2)$ are coherent and $C(\Gamma_1\cap\Gamma_2)$ is slender, then $C(\Gamma)$ is coherent.
\end{enumerate}
\end{corollary} 
\begin{proof}
It follows by the presentation of the graph product $G(\Gamma)$ and Coxeter group $C(\Gamma)$ that
\[
G(\Gamma)=G(\Gamma_1)*_{G(\Gamma_1\cap\Gamma_2)} G(\Gamma_2)\text{ and } C(\Gamma)=C(\Gamma_1)*_{C(\Gamma_1\cap\Gamma_2)} C(\Gamma_2).
\]
By Proposition \ref{Karrass} follows that $G(\Gamma)$ and $C(\Gamma)$ are coherent.
\end{proof}

\section{Proof of Theorem A}

We are now ready to prove Theorem A.

\begin{proof}
We prove this result by induction on $\#V=n$.
Suppose that $n=1$. Then $G(\Gamma)$ is a finitely generated abelian group and therefore coherent. 
Now we assume that $n>1$. If $\Gamma$ is not connected, then there exist induced proper disjoint subgraphs $\Gamma_1, \Gamma_2$ with $\Gamma=\Gamma_1\cup \Gamma_2$. It is evident from the presentation of the graph product that $G(\Gamma)=G(\Gamma_1)*G(\Gamma_2)$. By the induction assumption $G(\Gamma_1)$ and $G(\Gamma_2)$ are coherent. By Corollary \ref{amalgam} coherence is preserved under taking free products, hence $G(\Gamma)$ is coherent. If $\Gamma$ is not connected,  then we have the following cases:
\begin{enumerate}
\item If $\Gamma$ is complete, then $G(\Gamma)$ is a finitely generated abelian group and thus coherent.
\item If $\Gamma$ is not complete, then by Proposition \ref{Dirac} there exist proper induced subgraphs $\Gamma_1, \Gamma_2$ with the following properties: $\Gamma=\Gamma_1\cup\Gamma_2$ and $\Gamma_1\cap\Gamma_2$ is complete. It follows that
$G(\Gamma)=G(\Gamma_1)*_{G(\Gamma_1\cap\Gamma_2)} G(\Gamma_2)$. The groups $G(\Gamma_1), G(\Gamma_1)$ are coherent by the induction assumption. The group $G(\Gamma_1\cap\Gamma_2)$ is a finitely generated abelian group and hence slender. By Corollary \ref{amalgam} follows that $G(\Gamma)$ is coherent.
\end{enumerate}
\end{proof}

\section{Proof of Theorem B}

The proof of Theorem B is very similar to the proof of Theorem A.

\begin{proof} $ $\\
\begin{enumerate}
\item[to (i):] The proof is again by induction on $\#V=n$.
Suppose that $n=1$. Then $C(\Gamma)\cong \Z_2$  and therefore coherent. 
Now we assume that $n>1$. If $\Gamma$ is not connected, then there exist induced proper disjoint subgraphs $\Gamma_1, \Gamma_2$ with $\Gamma=\Gamma_1\cup \Gamma_2$. It is evident from the presentation of the Coxeter group that $G(\Gamma)=G(\Gamma_1)*G(\Gamma_2)$. By the induction assumption $C(\Gamma_1)$ and $C(\Gamma_2)$ are coherent. By Corollary \ref{amalgam} coherence is preserved under taking free products, hence $C(\Gamma)$ is coherent. If $\Gamma$ is connected,  then we have the following cases:
\begin{enumerate}
\item If $\Gamma$ is complete, then $C(\Gamma)$ is slender and by Corollary \ref{Coxeterslendercoherent} we know that slender Coxeter groups are coherent.
\item If $\Gamma$ is not complete, then by Proposition \ref{Dirac} there exist proper induced subgraphs $\Gamma_1, \Gamma_2$ with the following properties: $\Gamma=\Gamma_1\cup\Gamma_2$ and $\Gamma_1\cap\Gamma_2$ is complete. It follows, that
$C(\Gamma)=C(\Gamma_1)*_{C(\Gamma_1\cap\Gamma_2)} C(\Gamma_2)$. The groups $C(\Gamma_1), C(\Gamma_1)$ are coherent by the induction assumption and the group $G(\Gamma_1\cap\Gamma_2)$ is slender. By Corollary \ref{amalgam} follows that $C(\Gamma)$ is coherent.
\end{enumerate}
\item[to (ii):] Let $\Gamma$ be a Coxeter graph with a shape of a cycle of length $>3$. Then there exist proper subgraphs $\Gamma_1$ and $\Gamma_2$ with the following properties: $\Gamma_1, \Gamma_2$ are paths, $\Gamma=\Gamma_1\cup\Gamma_2$ and $\Gamma_1\cap\Gamma_2=(\left\{v_i, v_j\right\}, \emptyset)$ where $v_i$ and $v_j$ are disjoint vertices. The Coxeter group $C(\Gamma)$ has the following decomposition:
\[
G(\Gamma)=G(\Gamma_1)*_{G(\Gamma_1\cap\Gamma_2)} G(\Gamma_2)
\]
Since $\Gamma_1, \Gamma_2$ are trees, it follows by Theorem B(i) that $C(\Gamma_1), C(\Gamma_2)$ are coherent groups. Further, $C(\Gamma_1\cap\Gamma_2)$ is isomorphic to $\Z_2*\Z_2$ and therefore slender. By Corollary \ref{amalgam} we obtain that $C(\Gamma)$ is coherent. 
\end{enumerate}
\end{proof}

\section{Small right angled Coxeter groups}
We turn now to the proof of the following result.
\begin{proposition}
Let $\Gamma=(V, E)$ be a right angled Coxeter graph. 
\begin{enumerate}
\item[(i)] If $\# V\leq 5$ or $\#V=6$ and $\#E\leq 8$, then $C(\Gamma)$ is coherent.
\item[(ii)] If $\Gamma$ has the shape of the above graph, 
then the corresponding right angled 

\begin{figure}[h]
\begin{center}
\begin{tikzpicture}[scale=0.5, transform shape]
\draw[fill=black]  (0,0) circle (1pt);
\draw[fill=black]  (0,1) circle (1pt);
\draw[fill=black]  (0,2) circle (1pt);
\draw[fill=black]  (3,0) circle (1pt);
\draw[fill=black]  (3,1) circle (1pt);
\draw[fill=black]  (3,2) circle (1pt);
\draw (0,0)--(3,0);
\draw (0,0)--(3,1);
\draw (0,0)--(3, 2);
\draw (0,1)--(3,0);
\draw (0,1)--(3,1);

\draw (0,1)--(3,2);
\draw (0,2)--(3,0);
\draw (0,2)--(3,1);
\draw (0,2)--(3,2);

\end{tikzpicture}
\end{center}
\caption*{$K_{3,3}$}
\end{figure}
Coxeter group $C(\Gamma)$  is incoherent. 
\end{enumerate}
\end{proposition}
\begin{proof}
\begin{enumerate}
\item[to (i):]  If $\#V\leq 4$, then $\Gamma$ has no induced cycle of length $>3$ or $\Gamma$ is a cycle of length $4$. It follows by Theorem  B  that $C(\Gamma)$ is coherent. 

If $\#V=5$, then  (i) $\Gamma$ is not connected or (ii) $\Gamma$ has no cycle of length $> 3$ or (iii) $\Gamma$ is a cycle of length $5$ or (iv) $\Gamma$ has a shape of the one of the following graphs:

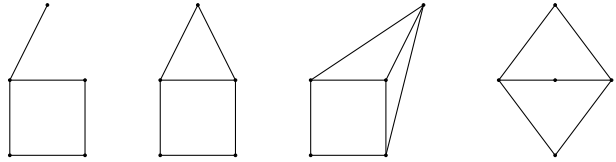
\begin{figure}[h]
\begin{center}
\begin{tikzpicture}[scale=0.5, transform shape]

\draw[fill=black]  (1,4) circle (1pt);
\draw[fill=black]  (5,4) circle (1pt);
\draw[fill=black]  (11,4) circle (1pt);
\draw[fill=black]  (14.5,4) circle (1pt);

\draw (0,0)--(2,0);
\draw (4,0)--(6,0);
\draw (8,0)--(10, 0);

\draw (0,2)--(2,2);
\draw (4,2)--(6,2);
\draw (8,2)--(10, 2);
\draw (13,2)--(16, 2);

\draw (0,0)--(0,2);
\draw (2,0)--(2,2);
\draw (4,0)--(4, 2);
\draw (6,0)--(6,2);
\draw (8,0)--(8,2);
\draw (10,0)--(10, 2);

\draw (0,2)--(1,4);
\draw (4,2)--(5,4);
\draw (6,2)--(5, 4);
\draw (8,2)--(11,4);
\draw (10,2)--(11,4);
\draw (10,0)--(11, 4);
\draw (14.5,0)--(13,2);
\draw (14.5,0)--(16,2);
\draw (13,2)--(14.5, 4);
\draw (16,2)--(14.5,4);

\draw[fill=black]  (0,0) circle (1pt);
\draw[fill=black]  (2,0) circle (1pt);
\draw[fill=black]  (4,0) circle (1pt);
\draw[fill=black]  (6,0) circle (1pt);
\draw[fill=black]  (8,0) circle (1pt);
\draw[fill=black]  (10,0) circle (1pt);
\draw[fill=black]  (14.5,0) circle (1pt);

\draw[fill=black]  (0,2) circle (1pt);
\draw[fill=black]  (2,2) circle (1pt);
\draw[fill=black]  (4,2) circle (1pt);
\draw[fill=black]  (6,2) circle (1pt);
\draw[fill=black]  (8,2) circle (1pt);
\draw[fill=black]  (10,2) circle (1pt);
\draw[fill=black]  (13,2) circle (1pt);
\draw[fill=black]  (14.5,2) circle (1pt);
\draw[fill=black]  (16,2) circle (1pt);

\end{tikzpicture}
\end{center}
\end{figure}

In the cases (i), (ii) and (iii) using Corollary \ref{amalgam} and Theorem B follows that $C(\Gamma)$ is coherent. If $\Gamma$ has a shape of the one of the above graphs then it is easy to verify that there exist induced subgraphs $\Gamma_1$ and $\Gamma_2$ such that $C(\Gamma)=C(\Gamma_1)*_{C(\Gamma_1\cap\Gamma_2)} C(\Gamma_2)$  and $C(\Gamma_1), C(\Gamma_2)$ are coherent and $C(\Gamma_1\cap\Gamma_2)$ is slender. By Corollary \ref{amalgam} we obtain that $C(\Gamma)$ is coherent.

If $\#V=6$ and $\#E\leq 8$, it follows with similar arguments as in the case $\#V=5$ that $C(\Gamma)$ is coherent. A table of connected graphs with six vertices is given in \cite{Cvetkovic}.

\item[to (ii):] The corresponding right angled Coxeter group $C(K_{3,3})$ is isomorphic to
\[
(\Z_2*\Z_2*\Z_2)\times(\Z_2*\Z_2*\Z_2).
\]
Since the kernel of the natural map $\Z_2*\Z_2*\Z_2\rightarrow \Z_2\times\Z_2\times\Z_2$ is a free group of rank $\geq 2$ it follows that $F_2\times F_2$ is a subgroup of $C(K_{3,3})$. The product $F_2\times F_2$ is incoherent and hence $C(K_{3,3})$ is incoherent.
\end{enumerate}
\end{proof}

\end{document}